\documentclass[12pt]{amsart}
\theoremstyle{definition}
\usepackage{geometry}
\usepackage{amssymb,latexsym}
\newcommand{\RNum}[1]{\lowercase\expandafter{\romannumeral #1\relax}}
\usepackage{amsthm}
\usepackage{amsmath}
\usepackage{hyperref}
\usepackage{makecell}
\usepackage{xcolor}
\newtheorem{example}{Example}[section]
\hypersetup{colorlinks=true,linkcolor=blue,urlcolor=blue,citecolor=red}
\newtheorem{theorem}{Theorem}[section]
\newtheorem{remark}{Remark}[section]

\newcommand{\beq}{\begin{equation}}
\newcommand{\eeq}{\end{equation}}

\newtheorem{definition}{Definition}[section]

\title[]{THE TELEPHONE EXCHANGE PROBLEM REVISITED: A combinatorial approach}

\author{Sithembele Nkonkobe}
\address{School of Mathematics, University of Witwatersrand, 2050 Wits, Johannesburg, South Africa }
\email{snkonkobe@gmail.com}

\date{\today}
\subjclass[2010]{}

\begin{document}
	
	\begin{abstract}
	In this study we revisit the telephone exchange problem. We discuss a generalization of the telephone exchange problem by discussing two generalizations of the Bessel numbers. We study combinatorial properties of these numbers and show how these numbers are related to the well-known Whitney numbers and Dowling numbers. 
	\end{abstract}
	\maketitle

	\section{Introduction}
	The Bessel polynomials are defined as the unique solutions of the differential equation (see~\cite{choi2003unimodality,grosswald1978bessel,krall1949new,yang2011bessel}) \begin{equation}x^2y^{\prime\prime}_n+(2x+2)y^\prime_n=n(n+1)y_n,
	\end{equation}
which is a Sturm-Liouville type second order differential equation.
		These polynomials seem to first appeared in the year 1929 paper \cite{bochner1929sturm}. The first systematic study of these polynomials seems to first been done by Krall and Frink in their 1949 paper \cite{krall1949new}. It is known that the  Bessel polynomial $y_n(x)$ may be written in the form $y_n(x)=\sum\limits_{k=0}^\infty B_{n,k}x^{n-k}$ where $B_{n,k}=\frac{(2n-k)!}{2^{n-k}k!(n-k)!}$, see, for instance, \cite{choi2003unimodality}. The numbers $B_{n,k}$ are known as the Bessel coefficients. For non-negative integers $n$ and $k$, the $(n,k)th$ Bessel number of the second kind $B(n,k)$ is defined as the Bessel coefficient $B_{k,2k-n}$ (see~\cite{choi2003unimodality}).
		
		The generating function of the Bessel numbers of the second kind, $B(n,k)$, is\cite{yang2011bessel}
		\begin{equation}
			\sum\limits_{n=0}^\infty B(n,k)\frac{z^n}{n!}=\frac{1}{k!}\begin{pmatrix}z+\frac{z^2}{2}\end{pmatrix}^k.
		\end{equation}
	 The Bessel numbers of the second kind $B(n,k)$ have the following combinatorial interpretation: They represent the total number of ways of partitioning the set $[n]=\{1,2,\ldots,n\}$ into $k$ non-empty subsets each subset having cardinality of at most two \cite{cheon2013generalized,choi2003unimodality}.
	 
	 	One of the problems of interest in enumerative combinatorics is the telephone exchange problem. The telephone exchange problem is as follows.\begin{definition}[Telephone exchange problem]\cite{riordan1959introduction}\label{definition:10} Suppose a telephone exchange has $n$ subscribers, in how many ways can the subscribers be connected satisfying the following conditions;
	
	a). subscribers may be connected in pairs only i.e. no conference calls are allowed,
	
	b). each of the $n$ subscribers is in one of two states; connected to exactly one other subscriber or not connected to any other subscriber at all.  \end{definition}

\begin{definition}
	We denote by $(i,j)$ to denote that subscribers $i$, and $j$ are connected with each other, where $i$ and $j$ are from the set $[n]$. We also denote by $\overline{i}$ to denote that  subscriber $i$ is not connected to any other subscriber.\end{definition}

We let $T_n$ be the total number of possible configurations in Definition~\ref{definition:10}. We refer to the numbers $T_n$ as telephone numbers. It turns out that the number of ways in which $n-k$ pairs out of $n$ subscribers can be connected by a telephone exchange is given by the Bessel numbers $B(n,k)$ \cite{cheon2013generalized}.
It is known that the telephone numbers $T_n$ have the following closed form, $T_n=\sum\limits^{n}_{k=0}B(n,k)$ where $B(n,k)$ are the Bessel numbers discussed above, see, for instance, \cite{jung2018r,mezo2014periodicity}. 
\begin{example} We demonstrate the  possibilities for two telephone numbers. For $T_2$ we have the following possibilities (\RNum{1}) $\overline{1}$,  $\overline{2}$,
		  (\RNum{2}) $(1,2)$. Hence $T_2=2$. For $T_3$ we have the following possibilities
				(\RNum{1}) $\overline{1}$,$\overline{2}$,$\overline{3}$,
				(\RNum{2}) $\overline{1}$, $(2,3)$,
				(\RNum{3}) $\overline{2}$, $(1,3)$,
				(\RNum{4}) $\overline{3}$, $(1,2)$. Hence, $T_3=4$.
		\end{example}
 In general we have the following \begin{table}[!h]
	\begin{tabular}{|c|c|c|c|c|c|c|c|c|c|c|}
		\hline
		$n$&0&1&2&3&4&5&6&7&8&$\cdots$\\\hline
		$T_n$&1&1&2&4&10&26&76&232&764&$\cdots$\\\hline
	\end{tabular}
	\vspace{3mm}
	
	\caption{Telephone numbers $T_n$.}
\end{table}

The telephone exchange problem has a reach history, see, for instance, \cite{cheon2013generalized,choi2003unimodality} and references therein. The telephone exchange problem seem to first appeared in Riordan's book\cite{riordan1959introduction}. There are several alternative interpretations of the telephone numbers. The telephone numbers $T_n$ also represent the Hosoya index (number of matchings) of the complete graph $K_n$\cite{yang2011bessel}. They also represent the number of involutions of the symmetric group $S_n$ 
 \cite{aigner2007course,bona2015handbook,solomon2005combinatorial,stanleyenumerativeVol2}. They also represent the total number of tableaux that can be formed from $n$ distinct elements\cite{knuth1998art}.

 The Hermite polynomials are defined in the following way\cite{banderier2002generating}
	\begin{equation}
		\sum\limits_{n=0}^\infty H_{e_n}(x)\frac{t^n}{n!}=exp\begin{bmatrix}
			xt-\frac{t^2}{2}
		\end{bmatrix}.
	\end{equation}
Also, the numbers $T_n$ are related to the hermit polynomials in the following way (see~\cite{banderier2002generating}) 
\begin{equation}\label{equation:30}
	T_n=i^nH_{e_n}(-i), \qquad i=\sqrt{-1}.
\end{equation}  

Equation~\ref{equation:30} means that the telephone numbers are the absolute values of the coefficients of the Hermite polynomials.  A problem that can be considered as a generalization of the telephone number problem is that of the restricted Bell numbers studied in \cite{mezo2014periodicity}, which restrict the number of elements each subsets may have.
\begin{definition}\label{definition:45}A generalization of the Bessel numbers of the second kind $B(n,k)$, which is the $r$-Bessel numbers of second kind $B_r(n,k)$ represents the number of partitions of $[n+r]$ elements into $k+r$ subsets of size one or two such that the first $r$ elements are in distinct blocks\cite{jung2018r}. We  denote by $B_r(n)$ the total number of all such possible configurations where $k$ runs from 0 to $n$. The $n+r$ elements denote $n+r$ subscribers. We will refer to the first $r$ subscribers the distinguished subscribers, and the rest as non-distinguished subscribers.\end{definition} The numbers $B_r(n,k)$ have the following generating function(see ~\cite{jung2018r}), 
\begin{equation}
		\sum\limits_{n=0}^\infty B_r(n,k)\frac{z^n}{n!}=\frac{1}{k!}(1+z)^r\begin{pmatrix}z+\frac{z^2}{2}\end{pmatrix}^k.\end{equation}

	The Bessel numbers $B_r(n,k)$ can alternatively be interpreted in the following way. Suppose $G$ is a complete semi-bipartite graph having bipartitions $V_1=\{1,2,\ldots,r\}$, and $V_2=\{r+1,r+2,\ldots,n+r\}$, then $B_r(n,k)$ is the number of matchings of $G$ having $n-k$ edges\cite{cheon2013generalized}. From Definition~\ref{definition:45} we have  $B_r(n)=\sum\limits_{k=0}^nB_r(n,k)$.
		We have (see \cite{cheon2013generalized})
	\begin{equation}
		\sum\limits_{n=0}^\infty B_r(n)\frac{z^n}{n!}=(1+z)^re^{z+z^2/2}.
	\end{equation}
 We will refer to the numbers $B_r(n)$ as the $r$-Bessel numbers. We denote by $i_r$ to indicate that the element $i$ is one of the first $r$ elements. 
 	\begin{example} 

We note that for $B_2(2)$ we have $\{1_2,2_2,1,2\}$. The eight partitions given by $B_2(2)$ are: 
\begin{enumerate}
	\item $\{1_2\}\{2_2\}\{1,2\}$,
	\item $\{1_2\}\{2_2\}\{1\}\{2\}$,
	\item $\{1_2,1\}\{2_2\}\{2\}$,
		\item $\{1_2,2\}\{2_2\}\{1\}$,
	\item $\{1_2,2\}\{2_2,1\}$,
	\item $\{1_2\}\{2_2,2\}\{1\}$,
\item $\{1_2\}\{2_2,1\}\{2\}$,
	\item $\{1_2,2\}\{2_2,1\}$.

\end{enumerate}
 \end{example}
 
 It turns out that $B_2(n)$ has the following sequence
 
 \begin{table}[!h]
 	\begin{tabular}{|c|c|c|c|c|c|c|c|c|c|c|}
 		\hline
 		$n$&0&1&2&3&4&5&6&7&8&$\cdots$\\\hline
 		$B_2(n)$&1&3&8&22&66&206&688&2388&8732&$\cdots$\\\hline
 	\end{tabular}
 	\vspace{3mm}
 	
 	\caption{The numbers $B_2(n)$.}
 \end{table}

  Another generalization of the Bessel Numbers related to the telephone exchange problem is the following (see~\cite{cheon2013generalized}) 
	\begin{equation}
		\sum\limits_{n=0}^\infty T_r(n,k)\frac{z^n}{n!}=e^{rz}\frac{(z+z^2/2)^k}{k!}.
	\end{equation}

\begin{definition}\label{definition:46}The numbers $T_r(n,k)$ represent the number of ways of selecting $n-k$ pairs (number of connections) out of $n+r$ subscribers where $r$ pre-specified subscribers cannot be connected with each other, however every conference call involves one of these $r$ subscribers\cite{cheon2013generalized}. We will refer to the $r$ subscribers as the distinguished subscribers, and the rest as non-distinguished subscribers. We denote by $T_r(n)$ the total number of all possible such configurations, where $k$ runs from 0 to $n$. \end{definition} The numbers $T_r(n)=\sum\limits_{k=0}^nT_r(n,k)$, have also been studied in \cite{cheon2013generalized}. We have (see \cite{cheon2013generalized})
\begin{equation}
	\sum\limits_{n=0}^\infty T_r(n)\frac{z^n}{n!}=e^{(r+1)z+z^2/2}.
\end{equation}
\begin{remark}
	Note that in the context of subsets the main difference between the subsets enumerated  by $B_n(r)$, and those enumerated by $T_r(n)$ is that on the set partitions given by $B_n(r)$ all the subsets have cardinality one or two. Whereas on the subsets enumerated by $T_r(n)$ subsets may have more than two elements as long as in each of those subsets one of the $r$ distinguished elements is contained in that subset.
\end{remark}

 In the telephone exchange problem discussed above it is assumed that all the subscribers are in different locations/one person per device at a time. Even for conference calls at least three devices/phones need to be connected in order for more than two people to be connected. In the above context in a single phone call it is not possible to have more that two subscribers participating in a phone call i.e. when only two devices are connected. In practice a short coming of the telephone exchange problem discussed above is that in reality it is possible in a single phone call to have multiple people participating in a conversation at each of the two ends of a single phone call. Motivated by this problem in this study we define a new type of numbers, that generalizes the telephone exchange problem above, to all for multiple subscribers to be at each end of a phone call. We further generalize to allow for subscribers to be in one of $\lambda$ locations, where each location constitute a different configuration. Also, we generalize the $r$-telephone numbers discussed above, by still requiring that $r$ fixed subscribers not to belong to the same subsets while at the same time allowing multiple subscribers to be involved in same call. The $r$ subscribers could be interpreted as people who have the same skill set, hence need to belong to different teams all the time or people at the same position who cant be in the same team. As expected the $r$-telephone numbers are a special case of the polynomials we define in this study. In this study we also show how our  generalizations of the Bessel numbers and $r$-telephone numbers are related to the Whitney numbers, and Dowling numbers.  

 \vspace{1cm}

	\section{Preliminaries}

	The $r$-stirling numbers ${n+r\brace k+r}_r$ are defined in the following way (see~\cite{broder1984r})
	\begin{equation}\label{equation:40}
		\sum\limits_{n=0}^\infty{n+r\brace k+r}_r\frac{z^n}{n!}=\begin{cases} \frac{e^{rz}(e^z-1)^k}{k!},\quad k\geq 0,\\0,\qquad \text{otherwise}.
		\end{cases}
	\end{equation}
The $r$-stirling numbers ${n+r\brace k+r}_r$ represent the number of ways of partitioning the set $[n+r]$ into $k+r$ non-empty subsets such that the first $r$ elements are in distinct subsets\cite{broder1984r}.

\begin{definition}[barred preferential arrangements] A barred preferential arrangement is an ordered set partition onto which a number of bars are inserted \cite{ahlbach2013barred, pippenger2010hypercube}.  
	\end{definition}	
Barred preferential arrangements seem to first appeared in \cite{pippenger2010hypercube}. 
	We note that $m$ bars induces $m+1$ sections between the bars. The sections are where the blocks of a set partition are ordered.  We now show two examples of barred preferential arrangements.
	\begin{example}We present two examples of barred preferential arrangements of the set $[4]$,
		
		\RNum{1}. $3\quad12|\quad|4,$
	
		\RNum{2}. $| 12\quad 3|\quad|4$.
	\end{example}
The barred preferential arrangement in $\RNum{1}$, has two bars hence, three sections. The first section the one to the left of the first bar from left to right has two blocks which are the singleton $\{3\}$ and the block $\{1,2\}$, the magnitude of the space between the elements indicates which elements are in the same block and which ones are not. The second section, the one between the two bars is empty. The third section has the singleton $\{4\}$. The barred preferential arrangement in $\RNum{2}$ has three bars hence four sections. The first section is empty. The second section has the two blocks $\{1,2\}$, and $\{3\}$. The third section is empty. The fourth section has the singleton $\{4\}$. The polynomials  interpreted as giving the number of barred preferential arrangements have extensively been studied in \cite{kargin2018higher,kargin2016generalization,nkonkobe2020combinatorial}

\begin{definition}\label{definition:4} The $r$-Whitney numbers of the second $\mathcal{W}_{m,r}(n,k)$ are defined as (see~\cite{cheon2012r}) \begin{equation}
	\mathcal{W}_{m,r}(n,k)=\frac{1}{m^kk!}\sum\limits_{t=0}^k(-1)^t\binom{k}{t}[m(k-t)+r]^n.
	\end{equation}
The $r$-Whitney numbers also satisfy the following recurrence relations (see~\cite{cheon2012r}) 
\begin{equation}
\mathcal{W}_{m,r}(n,k)=	\sum\limits_{i=k}^n\binom{n}{i}\mathcal{W}_{m,r-1}(i,k),
\end{equation}
\begin{equation}\label{equation:10}
	(mx+r)^n=\sum\limits_{k=0}^n\mathcal{W}_{m,r}(n,k)m^k(x)_k
.\end{equation}

The numbers $\mathcal{W}_{2,r}(n,k)$ are related to the $r$-bessel numbers in the following way (see \cite{jung2018r})
\begin{equation}
	\mathcal{W}_{2,r}(n,k)=\sum\limits_{i=0}^n{n\brace i}B_r(i,k)
.\end{equation}	The $r$-Whitney numbers of the second kind have the following combinatorial interpretation (see~\cite{gyimesi2019new}). The number 	$\mathcal{W}_{m,r}(n,k)$ is the number of colored partitions of $[n+r]$ into $k+r$ subsets such that:
	\begin{itemize}
		\item the first $r$ elements of $[n+r]$ are in distinct blocks, 
		\item none of the blocks containing the distinguished elements are colored (we will refer to these as distinguishable blocks), 
		\item 	 the smallest element in each non-distinguishable block is not colored,
		\item  all other elements in non-distinguishable blocks are colored with one of $m$ colors independently.
		
	\end{itemize}
\end{definition}
\begin{definition}
The $r$-Dowling polynomials $\mathcal{D}_{m,r}^x(n)$ are defined in the following two ways (see~\cite{corcino2018some,gyimesi2018comprehensive})
\begin{equation}\label{equation:4}
	\mathcal{D}_{m,r}^x(n)=\sum\limits_{k=0}^n\mathcal{W}_{m,r}(n,k)x^k,
\end{equation}
\begin{equation}\label{equation:220}
	\sum\limits_{n=0}^\infty\mathcal{D}_{m,r}^x(n)\frac{z^n}{n!}=exp\begin{bmatrix}
		rz+\frac{x(e^{mz}-1)}{m}
	\end{bmatrix}.
\end{equation}
For $x=1$ we denote  $\mathcal{D}_{m,r}^x(n)=\mathcal{D}_{m,r}(n)$.
\end{definition}A combinatorial interpretation of the numbers $\mathcal{D}_{m,r}(n)$ follows from that of  $\mathcal{W}_{m,r}(n,k)$ above, by adding the extra condition that all the $k$ non-distinguishable blocks are each independently  colored with one of $x$ colors. Now summing over $k$ from $0$ to $n$ gives the combinatorial interpretation of the $r$-Dowling polynomials $\mathcal{D}_{m,r}(n)$ (see~\cite{corcino2018some,gyimesi2018comprehensive}). We let $\mathbb{D}_{m,r}(n)$ be the set of these arrangements i.e $|\mathbb{D}_{m,r}(n)|=\mathcal{D}_{m,r}(n)$.
	\begin{definition}The $r$-Bell numbers $B_r^\lambda(n)$ are defined as (see~\cite{mezo2011r})
	\begin{equation}\label{equation:35}
			\sum\limits_{n=0}^\infty \mathcal{B}_r^\lambda(n)\frac{z^n}{n!}=exp[\lambda(e^z-1)+rz]. 
		\end{equation}
	The numbers $\mathcal{B}_r^\lambda(n)$ represent the number of ways of partitioning the set $[n+r]$ into $k+r$ subsets such that the first $r$ elements are in different subsets, where $k$ run from 0 to $n$.\cite{mezo2011r}. 
	\end{definition}

	\section{Generalized $r$-Bessel numbers}
	
 In this section we prove some combinatorial identities regarding the numbers $B_r(1;n)$.
 
 	\begin{definition}\label{definition:1}
 	Let $B_r(1,n)$  represent the number of ways of partitioning the set $[n+r]$ satisfying the following three conditions:
 	
 	\RNum{1}. $[n+r]$ is partitioned into $k+r$ subsets, such that the first $r$ elements are in distinct blocks. We will refer to the blocks containing the first $r$ elements as the distinguishable blocks. We will refer to the $r$ elements as the distinguishable elements, and the other $n$ elements as non-distinguishable elements.
 	
 	\RNum{2}. each of the $k+r$ blocks is in one of two states; it is either paired up with exactly one other block or is not paired with any block at all,
 	
 	\RNum{3} none of the $r$ distinguishable blocks can be paired with each other.
 \end{definition}
 
 We will refer to the partitions of the set $[n+r]$ given in Definition~\ref{definition:1} as $(B,r)$-partitions of $[n+r]$.

	
\begin{theorem}For $n,r\geq0$, we have
	\begin{equation}
		B_r(1,n)=\sum\limits_{k=0}^n\binom{n}{k} r^{n-k}B_0(1,k).
	\end{equation}
\end{theorem}
	\begin{proof}
		Given the set $[n+r]$. There are $\binom{n}{k}$ ways of choosing $k$ non-distinguishable elements. The number of $(B,r)$-partitions such that none of the $k$ non-distinguishable elements go into the $r$ distinguishable blocks is $B_0(1,k)$. The remaining $n-k$ non-distinguishable elements can go into the $r$ distinguishable blocks in $r^{n-k}$. Taking the product and summing over $k$ completes the proof.  
	\end{proof}

\begin{theorem}For $n\geq0$ and $r\geq1$ we have
	\begin{equation}
		B_r(1,n+1)=2rB_r(1,n)+\sum_{k=0}^{n}\binom{n}{k}2^kB_r(1,n-k).
	\end{equation}
\end{theorem}
\begin{proof}
In the $(B,r)$-partitions of $[n+r+1]$ we consider two cases.	

Case 1: The element $(n+1)$ is in one of the $r$ distinguished blocks. The number of  $(B,r)$-partitions of the set $[n+r]$ that can be formed is $B_r(1,n)$. Now the distinguishable block that $(n+1)$ belongs can be chosen in one of $r$ possible ways, we name it as block $B$. Block $B$ may be in one of two states either it is paired or not paired with another block. Hence, there are a total of $2rB_r(1,n)$ possibilities in this case.

Case 2: $(n+1)$ is in a non-distinguishable block, we name it as block $D$.  We name the block paired with block $D$, as block $C$. We choose in $\binom{n}{k}$ ways, elements that are in either block $D$ or Block $C$. Now the $k$ elements can be distributed among the two blocks in $2^k$ ways. With the remaining $n-k$ elements, $(B,r)$-partitions can be formed in $B_r(1,n-k)$ ways.   
\end{proof}
\begin{theorem}For $n\geq0$ and $r\geq1$ we have
	\begin{equation}
		B_r(1,n+1)=2rB_r(1;n)+B_{r+1}(1;n).
	\end{equation}
\end{theorem}
\begin{proof}
	In the $(B,r)$-partitions of $[n+r+1]$ we consider two cases.
	
	Case 1: The element $(n+1)$ is in one of the $r$-distinguishable blocks, Of which there are $r$ choices. The block having $(n+1)$  may or may not be paired with another block.
	
	Case 2: $(n+1)$ is in a non-distinguishable block. In a way the block containing $(n+1)$ can be viewed as the $(r+1)th$ distinguishable block, from the fact that it contains $(n+1)$. So the $(B,\lambda)$-partitions of the $[n+r]$ non-distinguishable other elements may be formed in $B_{r+1}(1;n)$ ways, where the block containing $(n+1)$ is the $(r+1)th$ distinguishable block.  
\end{proof}

	\section{Further generalized $r$-Bessel Numbers}

\subsection{Generalised Bessel numbers}
In this section we provide combinatorial identities concerning the numbers $B_r^\lambda(1;n)$.

Let $(n)_j$ denote the falling factorial $n(n-1)(n-2)\cdots(n-j+1)$.

	\begin{definition}\label{definition:3}
	Let  $B_r^\lambda(1;n)$ represent the number of ways of partitioning the set $[n+r]$ satisfying the following conditions:
	
	\RNum{1}. $[n+r]$ is partitioned into $k+r$ subsets of size one or two where $k$ runs from 0 to $n$, such that the first $r$ elements are in distinct blocks.
	
		\RNum{2}. each of the $k+r$ blocks is in one of two states; it is either paired up with exactly one other block or is not paired with any block at all,
	
		\RNum{3} none of the $r$ distinguishable blocks can be paired with each other,
	
	\RNum{4}. each of the $k$ non-distinguishable blocks are place in one of $\lambda$ sections,
	
	
\end{definition}

We will refer to the partitions of the set $[n+r]$ given in Definition~\ref{definition:3} as the $(B,r,\lambda)$-partitions of $[n+r]$. We let $\mathbb{B}_r^\lambda(1;n)$ denote the set of these partitions. The definition of the numbers $B_r^\lambda(1;n)$ above is a generalization of the numbers $T_r(n)$ in \cite{jung2018r}.
\begin{theorem}For $r,n,\lambda\geq0$
	\begin{equation}
		B_r^\lambda(1;n)=\sum\limits_{i=0}^r\binom{r}{i}(n)_iB_0^\lambda(1;n-i).
	\end{equation}
\end{theorem}
\begin{proof}
	On the $(B,r,\lambda)$-partitions of $[n+r]$, out of the $r$ distinguishable blocks, say there are $i$ of them that are having two elements. The $i$ distinguishable blocks can be chosen in $\binom{r}{i}$ ways. From the non-distinguishable elements $[n]$, $i$ elements can be chosen and placed on the $i$ chosen distinguishable blocks in $(n)_i$ ways.  Using the remaining $n-i$ non-distinguishable blocks $(B,r,\lambda)$-partitions can be formed in $B_0^\lambda(1;n-i)$. Summing over $r$ completes the proof.
\end{proof}
\begin{theorem}For $r\geq0$, and $n,\lambda\geq1$ we have
	\begin{equation}
		B_r^\lambda(1;n+1)=rB_r^\lambda(1;n)+\lambda B_r^\lambda(1;n)+n\lambda B_r^\lambda(1;n-1).
	\end{equation}
\end{theorem}
\begin{proof}
We partition the proof into three cases. 

Case 1: $(n+1)$ is in one of the $r$ distinguished blocks. There are $r$ possible ways of placing $(n+1)$. With the remaining $n$ elements $(B,r,\lambda)$-partitions can be formed in $B_r^\lambda(1;n)$ ways.

Case 2: $(n+1)$ is in one of the $\lambda$ sections, and is a singleton. There are $\binom{\lambda}{1}$ ways of choosing the section onto which $(n+1)$ is in. Then using the remaining $n$ elements there are $B_r^\lambda(1;n)$ ways of forming $(B,r,\lambda)$-partitions.

Case 3: $(n+1)$ is in one of the $\lambda$ sections, and is in a subset of size two. There are $\binom{\lambda}{1}$ ways to choose the section onto which $(n+1)$ belongs. We let $L$ denote the block onto which $(n+1)$ belongs.  There are ${n}$ ways of choosing the other element that is also part of $L$. Using the remaining $n-1$ non-distinguishable elements, $(B,r,\lambda)$-partitions can be formed in $B_r^\lambda(1;n-1)$ ways. 
	\end{proof}

\begin{theorem}For $r\geq0$, and $n,\lambda\geq1$ we have
	\begin{equation}
		B_r^\lambda(1;n+1)= rB_{r-1}^\lambda(1;n)+\lambda B_{r+1}^\lambda(1;n).
	\end{equation}
\end{theorem}
\begin{proof}
	In the $(B,r,\lambda)$-partitions of $[n+r+1]$ we consider two cases.
	
	Case 1: The element $(n+1)$ is in one of the $r$-distinguishable blocks. Of which there are $r$ possible ways. From the set $[n+r-1]$, the number of $(B,r,\lambda)$-partitions that can be formed is $B_{r}^\lambda(1;n)$ where the $rth$ distinguishable block is the one having $(n+1)$. 
	
	Case 2: $(n+1)$ is in a non-distinguishable block in one of the $\lambda$ sections. In a way the block containing $(n+1)$ can be viewed as the $(r+1)th$ distinguishable block. So the number of $(B,r,\lambda)$-partitions of  $[n+r]$ that can be formed is $B_{r+1}^\lambda(1;n)$, where the block containing $(n+1)$ is the $(r+1)th$ distinguishable block.  
\end{proof}
\begin{theorem}For $r\geq0$, and $n,\lambda\geq1$ we have
	\begin{equation}
		B_r^\lambda(1;n+1)= rB_{r-1}^\lambda(1;n)+\lambda\sum\limits_{i=0}^{r+1}\binom{r+1}{i}(n)_iB_r^\lambda(1;n-i).
	\end{equation}
\end{theorem}
\begin{proof}
	In the $(B,r,\lambda)$-partitions of $[n+r+1]$ we consider two cases.
	
	Case 1:$(n+1)$ is on one of the $r$ distinguishable blocks. The are $B_{r-1}^\lambda(1;n)$, $(B,r,\lambda)$-partitions that can be formed from $[n+r-1]$ where the distinguishable block containing $(n+1)$ is viewed as the $rth$ distinguishable block.
	
	Case 2: $(n+1)$ is in a non-distinguishable block in one of the $\lambda$ sections. There are $\lambda$ ways in which the sections onto which $(n+1)$ belongs can be chosen. The block having $(n+1)$ can be viewed as the $(r+1)th$ distinguishable block. Now say from the $r+1$ distinguishable blocks $i$ of them have cardinality two. The $i$ blocks can be chosen in $\binom{r+1}{i}$ ways. From the $n$ non-distinguishable elements, $i$ elements can be chosen and placed on the distinguishable blocks having cardinality two in $(n)_i$ ways. Hence, the total number of possibilities in this case is $\lambda\sum\limits_{i=0}^{r+1}\binom{r+1}{i}(n)_iB_r^\lambda(1;n-i)$.
\end{proof}

\vspace{5cm}

\begin{definition}\label{definition:100}The numbers $\tilde{B}_r^\lambda(1;n)$ represent the number of partitions of $[n+r]$ into $k+r$ subsets (where $k$ runs from 0 to $n$), such that
	\begin{itemize}
		\item the first $r$ elements of $[n+r]$ form singletons (we will refer to these as distinguishable blocks),
  \item  the elements $[r+1,n+r]$ are partitioned into $k$ subsets,
	\item each of the $k$ subsets either goes into one of $\lambda$ sections, or is paired up with one of $r$ distinct singletons,
    \item for those blocks that go into the $\lambda$ sections within the same section each subset is either paired with another subset or not paired with any at all. 
	\end{itemize}

	We let $\tilde{\mathbb{B}}_r^\lambda(1;n)$ be the set of these partitions i.e. $|\tilde{\mathbb{B}}_r^\lambda(1;n)|=\tilde{B}_r^\lambda(1;n)$.
\end{definition}

\begin{table}[hbt!]
	\begin{tabular}{|c|c|}\hline
	On $x\in \mathbb{D}_{2,r}^\lambda(n)$	& On $x\in \tilde{\mathbb{B}}_r^\lambda(1;n)$ \\  \hline $x$ has $r$ distinguishable blocks&  $x$ has $r$ distinguishable blocks
	\\\hline
	  \begin{minipage}{5cm}
	 	$x$ has $k$ non-distinguishable blocks for fixed $0\leq k\leq n$\end{minipage}& \begin{minipage}{5cm}
	 	 $x$ has $k$ non-distinguishable blocks for fixed $0\leq k\leq n$\end{minipage}\\\hline
	\begin{minipage}{5cm}
		the non-distinguishable elements are either in the distinguishable $r$ blocks or\end{minipage} &\begin{minipage}{5cm}
		the non-distinguishable blocks are either paired with one of the $r$ distinguishable blocks or \end{minipage} 
	\\\hline \begin{minipage}{5cm}
		
	 in one of the $\lambda$ sections.\end{minipage}& in one of the $\lambda$ sections.
 \\\hline
 \begin{minipage}{5cm}
	In the $\lambda$ sections on each non-distinguishable block all non minimal elements are colored with one of two colors, where the minimum element in each block is always colored with the first color.\end{minipage}& \begin{minipage}{5cm}
	In the $\lambda$ sections each block is either paired up with another block or none.  \end{minipage} \\\hline
	\end{tabular}\caption{}\label{table:3}
\end{table}	

From this clearly there is a one to one correspondence between elements of  $\mathbb{D}_{2,r}^\lambda(n)$ and those of $ \tilde{\mathbb{B}}_r^\lambda(1;n)$ i.e. $|\mathbb{D}_{2,r}^\lambda(n)|=|\tilde{\mathbb{B}}_r^\lambda(1;n)|$.

\begin{theorem}For $n,r\geq0$, and $\lambda\geq1$ we have
	\begin{equation}
\mathcal{D}_{2,r}^\lambda(n+1)	=r\sum\limits_{k=0}^n{n\brace k}{B}_r^\lambda(1;k)+2\lambda\sum\limits_{k=0}^n{n\brace k} {B}_{r+1}^\lambda(1;k). 
\end{equation}\end{theorem}
\begin{proof}
	We split the proof into two cases.

		Case 1:  With the elements $[n+r]$, $k+r$ blocks may be formed in ${n\brace k}$ ways, having the $r$ singletons. In this case the element $(n+1)$ can go into one of the $r$ distinguishable blocks in $r$ ways, say $(n+1)$ goes into  $r^\prime$. There are  ${B}_r^\lambda(1;k)$ ways for the $k$ blocks for each to either be in one of the $r$ distinguishable blocks or in one of $\lambda$ sections. Where on the $\lambda$ sections on each block the elements are colored with one of two colors, and on each block the minimum element is always colored with the first color.   On the $(B,r,\lambda)$-partitions of $[n+r]$ this corresponds to each block being paired or not with another block  in the $\lambda$ sections. On the $(B,r,\lambda)$ partitions, the block $r^\prime$ is interpreted as the block having $(n+1)$ being paired with $r^\prime$. Summing over $k$ completes the number of possibilities in this case.

		Case 2: When $(n+1)$ is in one of the $\lambda$ sections. Call it block $\mathbb{C}$. In this scenario $\mathbb{C}$ can be regarded as the $(r+1)th$ distinguishable block. Elements of block $\mathbb{C}$ can be partitioned into two subsets, those colored with first color, and those colored with the second color. On the  $(B,r,\lambda)$ partitions this corresponds to the block having $(n+1)$ being paired or not with another block. There are  ${B}_{r+1}^\lambda(1;k)$ ways for each of $k$ blocks to be either be in one of the $r+1$ distinguishable blocks or in one of $\lambda$ sections. Where on the $\lambda$ sections on each block the elements are colored with one of two colors, and on each block the minimum element is always colored with the first color. 

\end{proof}

 Using Table~\ref{table:3} in comparison with Definition~\ref{definition:4} the number of partitions of $[n+r]$ such that the first $r$ elements form distinguishable blocks, and the elements $[r+1,n+r]$ form $k$ blocks. Each of those $k$ blocks is either  paired with one of the $r$ blocks or in one of $\lambda$ sections. Where on the $\lambda$ sections each block is either paired or not paired with another block is given by $\mathcal{W}_{2,r}(n,k)\lambda^k$. In a variation of  Definition~\ref{definition:4} these correspond to the following respectively; the first $r$ elements of $[n+r]$ are in distinguishable blocks, and the elements $[r+1,n+r]$ form $k$ non-distinguishable blocks.
  The smallest element in each non-distinguishable block is always colored with the first color out of two colors,
 	 all other elements in non-distinguishable blocks are colored with one of two colors independently. Summing over $k$ we get the following theorem.

\begin{theorem}For $n,r\geq0$, and $\lambda\geq1$ we have
	\begin{equation}\label{equation:34}
		\tilde{B}_r^\lambda(1;n)=\sum\limits_{k=0}^n\mathcal{W}_{2,r}(n,k)\lambda^k.
	\end{equation}
\end{theorem}

\begin{theorem}For $n\geq0$, and $r,\lambda\geq1$ we have
	\begin{equation}
		\tilde{B}_r^\lambda(1;n)=\sum\limits_{k=0}^n\sum\limits_{i=k}^n\binom{n}{i}\mathcal{W}_{2,r-1}(i,k)\lambda^k.
	\end{equation}
\end{theorem}
\begin{proof}In forming the partitions $\tilde{\mathbb{B}}_r^\lambda(1;n)$ we fix one distinguishable block from the $r$ distinguishable blocks, call it $r^\prime$. Then $n-i$ elements from $[r+1,n+r]$ that form a block that is paired with $r^\prime$ can be chosen in $\binom{n}{i}$ ways. The remaining $i$ elements can form $k$ blocks (where $k$ runs from o to $n$). The $k$ blocks can either be paired with the other $r-1$ distinguishable blocks or be in one of $\lambda$ sections, each block on the $\lambda$ sections may be paired or not with another block in $\sum\limits_{i=k}^n\binom{n}{i}\mathcal{W}_{2,r-1}(i,k)\lambda^k$ ways. 
	
\end{proof}
\begin{theorem}For $n,r\geq0$, and $\lambda\geq1$ we have
	\begin{equation}
		\tilde{B}_r^\lambda(1;n)=\sum\limits_{k=0}^n\begin{bmatrix}\sum\limits_{i_1=k}^n\binom{n}{i_1}\sum\limits_{i_2=k}^{i_1}\binom{i_1}{i_2}\sum\limits_{i_3=k}^{i_2}\binom{i_2}{i_3}\cdots\sum\limits_{i_{r}=k}^{i_{r-1}}\binom{i_{r-1}}{i_r}\end{bmatrix}\mathcal{W}_{2,0}(i_r,k)\lambda^k.
	\end{equation}
\end{theorem}
\begin{proof} We let the $r$ distinguishable blocks be denoted by  $r_1,r_2,\ldots,r_r$. In forming the partitions $\tilde{\mathbb{B}}_r^\lambda(1;n)$ we let $i_1$ be the number of elements of $[r+1,n+r]$ that form the $k$ non-distinguishable blocks. From the $i_1$ elements, let $i_2$ be the number of elements that not in the block paired with $r_1$. The $i_2$ elements can be chosen in $\binom{i_1}{i_2}$ ways. From the $r_2$ elements we choose in $\binom{i_2}{i_3}$ ways elements that are not in the block paired with $r_2$. We keep going... Finally we choose $i_r$ elements that are not in the block paired with $r_r$ in $\binom{i_{r-1}}{i_r}$ ways. With the $i_r$ elements $k$ blocks that go into $\lambda$ sections, where on the $\lambda$ sections. Each block is paired with another block or none can be done in $\mathcal{W}_{2,0}(i_r,k)\lambda^k$ ways.
\end{proof}

\section{Telephone numbers}

	\subsection{Higher order Telephone numbers}
		\begin{definition}\label{definition:1bb}
		Let  $T_r^\lambda(n)$ represent the number of ways of selecting $n-k$ pairs (number of connections) out of $n+r$ subscribers where $r$ pre-specified subscribers cannot be connected with each other, however every conference call involves one of these $r$ subscribers (where $k$ runs from 0 to $n$) such that 
		
		\RNum{1}. the subsets(connections) that do not include any of the $r$ subscribers go into $\lambda$ sections.
		
		\RNum{2}. these subsets that go into the $\lambda$ sections are of size one or two.
		
		
	\end{definition} We refer to the partitions in Definition~\ref{definition:1bb} as $(T,r,\lambda)$ partitions.

	\begin{theorem}For $n,r\geq0$, and $\lambda\geq1$ we have
		\begin{equation}
			T^\lambda_r(n+1)=(r+\lambda)T^\lambda_r(n)+n\lambda T^\lambda_r(n-1).
		\end{equation}
	\end{theorem}
\begin{proof}We split the proof into two cases. 

	Case 1: the $(n+1)th$ subscriber is connected with one of the $r$ subscribers, this can happen in $r$ ways. With the remaining $[n+r]$ subscribers $(T,r,\lambda)$ partitions can be formed in $T^\lambda_r(n)$ ways.
	
	Case 2: $(n+1)$ is in one of the $\lambda$ sections. There are $\lambda$ ways of selecting a section to which $(n+1)$ belongs to. If in the section $(n+1)$ is not connected to another subscriber then there are $T^\lambda_r(n)$ ways of forming $(T,r,\lambda)$ partitions with the other $[n+r]$ elements. If in the section $(n+1)$ is connected with another subscriber then there are $n$ ways of selecting the subscriber, and then there are $T^\lambda_r(n)$ ways of forming $(T,r,\lambda)$ partitions with the other $[n+r-1]$ elements.  
	\end{proof}

\begin{theorem}For $n,r\geq0$, and $\lambda\geq1$ we have
	\begin{equation}\label{equation:300}
		T^\lambda_r(n)=\sum\limits_{k=0}^n\binom{n}{k}\tilde{B}_0^\lambda(1;k)r^{n-k}.
	\end{equation}
\end{theorem}
\begin{proof}
	From $[n]$ we can choose in $\binom{n}{k}$ ways $k$ subscribers that are not in the $\lambda$ sections. The $k$ subscribers may each be pairs with the $r$ distinguishable subscribers in $r^{n-k}$ ways. The remaining $n-k$ subscribers may go into the $\lambda$ sections, where in the section each subscriber is either connected or not with another subscriber in $\tilde{B}_0^\lambda(1;k)r^{n-k}$. \end{proof}

		\subsection{Higher Order Generalized Telephone numbers}
	The following is a generalization of Definition~\ref{definition:1bb} above.
	\begin{definition}\label{definition:2bb}
		Let  $\tilde{T}_r^\lambda(n)$ represent the number of ways of partitioning $[n+r]$ into $k+r$ subsets such that the first $r$ elements of $[n+r]$ are in distinct blocks such that

		\RNum{1}. the blocks that go into $\lambda$ sections. 
		
		\RNum{2}. each subset on the $\lambda$ sections is either paired with another block or not,
		
		\RNum{3}. each of the $r$ blocks may be paired with several blocks, 
		
		
	\end{definition}
	We will refer to the partitions in described in Definition~\ref{definition:2bb} as the $(\tilde{T},r,\lambda)$ partitions. 
	\begin{remark}
		The main difference between the numbers $\tilde{T}_r^\lambda(n)$ and the numbers $\tilde{B}_r^\lambda(1;n)$ each $r$ blocks in $\tilde{B}_r^\lambda(1;n)$ may be paired with several other blocks, whereas in $\tilde{B}_r^\lambda(1;n)$ each of the $r$ blocks may be paired with a maximum of one block.\end{remark}

	\begin{theorem}For $n,r\geq0$, and $\lambda\geq1$ we have
		\begin{equation}\label{equation:33}
			\tilde{T}_r^\lambda(n)=\sum\limits_{k=0}^n\binom{n}{k}\mathcal{B}_0^r(k)\mathcal{D}_{2,0}^\lambda(n-k).
		\end{equation}
	\end{theorem}
\begin{proof}In forming the $(\tilde{T},r,\lambda)$ partitions of $[n+r]$, we can choose in $\binom{n}{k}$ ways elements from $[r+1,n+r]$.The blocks may be paired independently with one of the $r$ blocks in $\mathcal{B}_0^r(k)$ ways. The remaining $n-k$ elements may form blocks that go into $\lambda$ sections in $\mathcal{D}_{2,0}^\lambda(n-k)$ ways, where in the $\lambda$ sections each block is either paired or not paired with any other block.
	
\end{proof}
	\begin{theorem}For $n,r\geq0$, and $\lambda\geq1$ we have
		\begin{equation}
				\tilde{T}_r^\lambda(n+1)=r	\tilde{T}_r^\lambda(n)+\lambda\sum\limits_{k=0}^n\binom{n}{k}2^k	\tilde{T}_r^\lambda(n-k).
		\end{equation}
	\end{theorem}
\begin{proof} On the $(\tilde{T},r,\lambda)$
	partitions of $[n+r+1]$ there are two cases to consider. 
	
	Case 1: $(n+1)$ is in a block paired with one of the $r$ blocks. There are $r$ choices. With the remaining elements $ (\tilde{T},r,\lambda)$ partitions may be formed in $\tilde{T}_r^\lambda(n)$.
	
	Case 2: $(n+1)$ is part of a block in one of the $\lambda$ sections. Call it block $S$. From $[r+1,n+r]$, $k$ elements can be chosen in $\binom{n}{k}$ ways that are either in block $S$ or in a block paired with $S$. With the remaining $r+n-k$ elements $ (\tilde{T},r,\lambda)$ partitions can be formed in $\tilde{T}_r^\lambda(n-k)$ ways.
\end{proof}
\begin{theorem}For $n,r\geq0$, and $\lambda\geq1$ we have
	\begin{equation}
		\mathcal{D}_{2,0}^\lambda(n)=\sum\limits_{k=0}^n\binom{n}{k}(-1)^k\mathcal{B}_0^r(k)\tilde{T}_r^\lambda(n-k).
	\end{equation}
\end{theorem}
\begin{proof}Let $\mathcal{M}_k$ be the number of elements of $[r+1,n+r]$ that form blocks which are paired with the $r$ singletons. There are $\binom{n}{k}$ ways to choose the $k$ elements, and $|\mathcal{M}_k|=\binom{n}{k}\mathcal{B}_0^r(k)\tilde{T}_r^\lambda(n-k)$. Applying the inclusion/exclusion principle completes the proof. 
	\end{proof}

\section{A generating function approach}

By \eqref{equation:220} we have 

\begin{equation}\label{equation:9}
	\sum\limits_{n=0}^\infty \tilde{B}_{r(1+\lambda)}^\lambda(1;n)\frac{z^n}{n!}=e^{(1+\lambda)rz}exp[{\frac{\lambda(e^{2z}-1)}{2}}].
\end{equation}
Combining \eqref{equation:9}, and \eqref{equation:4} we have the following $Dobi{\acute{n}}ski$ type formula.
\begin{theorem} For $n\geq 0$ we have 
	\begin{equation}
	\tilde{B}_{r(1+\lambda)}^\lambda(1;n)=e^{-\frac{\lambda}{2}}\sum\limits_{m=0}^\infty \frac{\lambda^m}{2^mm!}\sum\limits_{k=0}^n\mathcal{W}_{2,r\lambda}(n,k)2^{k+1}.		
	\end{equation}
\end{theorem}

From \eqref{equation:9} we have

\begin{equation}
	\tilde{B}_{r(1+\lambda)}^\lambda(1;n)=\sum\limits_{k=0}^n\sum\limits_{i=0}^k\frac{\lambda^k}{2^k}(-1)^{k-i}(2i+(1+\lambda)r)^n.
\end{equation} This implies that following.
\begin{theorem}For $n\geq 0$ we have

	\begin{equation}\tilde{B}_{r(1+\lambda)}^\lambda(1;n)=	\sum\limits_{k=0}^n\sum\limits_{i=0}^k\sum\limits_{j=0}^n\sum\limits_{m=0}^{n-j}\binom{n-j}{m}(-1)^{k-i}2^{j-k}i^jr^{n-j}\lambda^{m+k}.\end{equation}\end{theorem}

Combining \eqref{equation:4}, \eqref{equation:35}, and \eqref{equation:33} we have

\begin{equation}
	\sum\limits_{n=0}^\infty\tilde{T}_r^\lambda(n)\frac{z^n}{n!}=e^{r(e^z-1)+\frac{\lambda(e^{2z}-1)}{2}}.
\end{equation}

We have 
\begin{equation}\label{equation:23}
	e^{rz}exp\begin{bmatrix}
		r\lambda(e^z-1)+\lambda\frac{e^{2z}-1}{2}\end{bmatrix}=e^{rz}exp\begin{bmatrix}
			\lambda r(e^z-1)
		\end{bmatrix}\sum\limits_{l=0}^\infty\frac{\lambda^l(e^z-1)^l(e^z+1)^l}{2^ll!}.
\end{equation}

Combining \eqref{equation:23} with \eqref{equation:40} we have the following theorem.
\begin{theorem}For $n,\lambda\geq0$ we have
	\begin{equation}
	\sum\limits_{i=0}^n\tilde{T}_r^\lambda(i)r^{n-i}=\sum\limits_{t=0}^n\sum\limits_{l=0}^n\sum\limits_{m=0}^l\frac{\lambda^{t+l}r^t}{2^l}\binom{l}{m}\binom{l+t}{t}{n+r+m\brace l+t+r+m}_{r+m}.
	\end{equation}
\end{theorem}
\section*{Acknowledgements}


\end{document}